\documentclass[11pt]{article}
\usepackage{amsmath,amssymb,amsfonts,amsthm}
\usepackage{fullpage}
\usepackage{mathtools}
\usepackage{bm}
\usepackage{graphicx}
\usepackage{verbatim}
\usepackage{color}
\usepackage{hyperref}
\usepackage{algorithm}
\usepackage{algpseudocode}
\usepackage{caption}
\usepackage{subcaption}
\usepackage{float}
\usepackage{amsthm}

\makeatletter
\def\th@remark{%
  \normalfont 
  \thm@headfont{\bfseries} 
  \thm@notefont{\bfseries}%
}
\makeatother

\theoremstyle{remark}
\newtheorem*{remark}{Remark}

\newtheorem{theorem}{Theorem}

\begin{document}
	
	\title{Construction of the Nearest Nonnegative Hankel Matrix for a Prescribed Eigenpair}
	\author{
		Prince Kanhya\thanks{Corresponding Author: Postdoctoral Fellow, IIT Guwahati, Assam, India. Email: \texttt{pkanhya77@gmail.com}} 
		\and 
		Udit Raj\thanks{Ph.D., Shiv Nadar Institution of Eminence, Deemed to be University, Delhi-NCR, India. Email: \texttt{ur376@snu.edu.in}} 
	}
	\date{\today}
	\maketitle
\begin{abstract}
 We study the problem of determining whether a prescribed eigenpair $(\lambda,x)$
	can be made an exact eigenpair of a nonnegative Hankel matrix through the smallest
	possible structured perturbation. The task reduces to check the feasibility of a
	set of linear constraints that encode both the Hankel structure and entrywise
	nonnegativity. When the feasibility set is nonempty, we compute the minimum-norm
	perturbation $\Delta H$ such that $(H+\Delta H)x=\lambda x$. When no such perturbation
	exists, we compute the nearest nonnegative Hankel matrix in a residual sense by
	minimizing $\|(H+\Delta H)x-\lambda x\|_{2}$ subject to the imposed constraints.
	Because closed--form formulas for the structured backward error are generally
	unavailable, our method provides a fully numerical and optimization-based
	framework for evaluating eigenpair sensitivity under nonnegativity-preserving
	Hankel perturbations. Numerical examples illustrate both feasible and infeasible cases.
\end{abstract}


	\section{Introduction}
	
	Hankel matrices arise naturally in a wide range of mathematical and engineering applications, including moment problems, signal processing, control theory, time-series analysis, system identification, and low-rank approximation of structured data. A matrix $H \in \mathbb{C}^{n \times n}$ is called a Hankel matrix if its entries remain constant along each anti-diagonal, i.e., $H(i,j) = h_{i+j-1}$. Owing to this structural constraint, the space of $n \times n$ Hankel matrices is a $(2n-1)$-dimensional linear subspace of $\mathbb{C}^{n \times n},$  see \cite{ChuGolub2005, Park1999,AlHomidan2007,Knirsch2020} for more information. In many applications, an additional positivity requirement is imposed: for instance, in classical and truncated moment problems, the associated Hankel matrix must be positive semidefinite; in spectral estimation and power spectral density (PSD) signal processing, covariance Hankel matrices are often required to be entrywise nonnegative or to preserve certain positivity patterns, see \cite{AlHomidan2007, Cifuentes2019, RanTeng2024, AlHomidan2003,Beckermann2007} for further work on Hankel matrices and nonnegativity structures.
 Such settings naturally motivate the study of nonnegativity-preserving structured perturbations, where only Hankel matrices with nonnegative entries are admissible.
	
	The numerical analysis of eigenpairs of structured matrices has a long history. Classical backward error results for unstructured eigenvalue problems date back to Wilkinson and Stewart and have been extended to structured settings including Toeplitz, Hankel, circulant, and palindromic structures \cite{ChuGolub2005, Park1999, Knirsch2020,AhmadKanhya2021, AlHomidan2021}. More recently, structured backward errors have been studied for polynomial eigenvalue problems, rank-constrained approximations, and matrix pencils with symmetries or convex cone constraints \cite{Cifuentes2019, Guglielmi2018,AhmadKanhya2020}. However, for Hankel matrices with entrywise nonnegativity constraints, the analysis becomes significantly more subtle: the feasible perturbations belong simultaneously to a linear subspace (Hankel structure) and a convex cone (nonnegativity), resulting in a nontrivial intersection geometry. To our knowledge, there is no closed-form expression for the smallest nonnegative Hankel perturbation that makes a given vector an exact eigenvector, for more information on Hankel matrices and related problems, see \cite{ChuGolub2005, Park1999, AlHomidan2007, Knirsch2020, Cifuentes2019} and the references therein.

The motivation for this study is to understand how small, structured perturbations can restore feasibility in Hankel matrices arising from approximate data while preserving physical or mathematical constraints. In applications such as moment sequence reconstruction, PSD signal processing, and low-rank Hankel approximations, one often encounters Hankel matrices that are nearly, but not exactly, positive semidefinite or nonnegative, see \cite{ ChuGolub2005, Park1999, Knirsch2020, AlHomidan2003} for more on various applications. The eigenpairs of such matrices encode essential information, including frequencies, vibration modes, and system dynamics, making their accurate computation critical. When the prescribed eigenpair cannot be realized exactly by a nonnegative Hankel matrix, a residual-minimizing adjustment provides the closest feasible solution and quantifies the structured backward error, see  \cite{Dopico2019, ShanWei2025,AhmadKanhya20211, AlHomidan2005,Fazzi2021} for related works. Advances in convex optimization, structured least-squares methods, and large-scale algorithms now make it possible to compute these corrections efficiently for moderately sized problems.

Further, it is natural to ask whether the underlying feasibility problem is always solvable. The answer is negative, even in very simple situations. Consider the Hankel matrix $H=0\in\mathbb{R}^{3\times 3}$, the scalar $\lambda=-1$, and the vector $x=[1,1,1]^T$. Suppose that a nonnegative Hankel matrix $\widehat H\ge 0$ exists such that $\widehat H x=\lambda x$. Since $\widehat H$ is entrywise nonnegative, the Perron--Frobenius theory implies that its spectral radius is a nonnegative eigenvalue with a nonnegative eigenvector. Hence $\widehat H$ cannot admit a negative eigenvalue with a strictly positive eigenvector, contradicting $\widehat H x = -x$ with $x>0$. 

For completeness we also give a direct, elementary contradiction. Any $3\times 3$ Hankel matrix has the form
\[
\widehat H=\begin{bmatrix}
h_1 & h_2 & h_3\\[1ex]
h_2 & h_3 & h_4\\[1ex]
h_3 & h_4 & h_5
\end{bmatrix},
\qquad h_i\ge 0.
\]
The eigenpair equation $\widehat H x=\lambda x$ with $\lambda=-1$ and $x=[1,1,1]^T$ gives
\[
(h_1+h_2+h_3,\; h_2+h_3+h_4,\; h_3+h_4+h_5)^T = (-1,-1,-1)^T.
\]
Each component on the left is a sum of nonnegative numbers, hence is nonnegative, whereas the right-hand side components are strictly negative. This is impossible. Therefore no nonnegative Hankel perturbation $\Delta H$ exists such that $(H+\Delta H)x=\lambda x$ for this prescribed eigenpair. This example demonstrates that the feasibility problem for exact eigenpair realization may fail, motivating the need for a residual-minimizing correction framework when no exact nonnegative Hankel matrix satisfies the prescribed eigenpair constraints.

\section{Problem Statement}
In this work, we mainly study the following problem: 
Given an arbitrary Hankel matrix $H \in \mathbb{R}^{n\times n}$, we seek the
nearest nonnegative Hankel matrix $\widehat H$ such that a prescribed
eigenpair $(\lambda,x)\in\mathbb{C}\times\mathbb{C}^n$ is (exactly or
approximately) preserved. Formally, we consider the optimization problem
\[
\min_{\Delta H}\;\|\Delta H\|_F
\quad\text{subject to}\quad 
\widehat H := H+\Delta H \ge 0,\;\;
\widehat H\ \text{Hankel}.
\]
When exact feasibility is possible, the goal is to find the minimum--norm
Hankel perturbation $\Delta H$ producing a nonnegative matrix $\widehat H$
satisfying
\[
\widehat H x = \lambda x.
\]
If exact feasibility is impossible under the nonnegativity and Hankel
constraints, we instead compute the nearest nonnegative Hankel matrix
$\widehat H$ minimizing the eigenpair residual
\[
\|\widehat H x - \lambda x\|_2.
\]
This leads to two structured convex optimization problems corresponding to
the exact and inexact correction stages of the algorithm, and solvable via modern convex optimization technique.
    
 
	Unlike classical unstructured formulas, there is no closed-form expression for the minimum nonnegative Hankel perturbation that satisfies the exact eigenpair constraint. Hence, our strategy is strictly numerical: we cast the feasibility problem as a linear program and when feasible, obtain the minimal Frobenius-norm solution via a constrained quadratic program. When infeasible, we solve a cone-constrained optimization problem to compute the nearest admissible matrix producing the smallest eigenpair residual. This unified computational framework hence yields a practical and theoretically motivated tool for the study of eigenvalue sensitivity under nonnegativity-preserving Hankel perturbations, both for exact and inexact settings.
	
	The contributions of the paper can be summarized as follows. First, we formulate the nonnegativity-constrained Hankel eigenpair backward error problem in a linear–convex framework. Second, we derive feasibility conditions and express the general form of admissible Hankel perturbations using a structure matrix $S$ that parameterizes all Hankel matrices. Third, we propose numerical algorithms based on linear programming and constrained least squares to compute either the exact structured backward error or the nearest admissible matrix when feasibility fails. Finally, we demonstrate the effectiveness and versatility of the proposed approach through numerical experiments. Two representative examples are presented: one in which the prescribed eigenpair is exactly realizable by a nonnegative Hankel matrix, and another where exact realization is impossible but residual-minimizing solution is obtained. Complementing these case studies, we also present a selection of graphical results that describe the performance of the algorithm: these include comparisons between the minimum-norm and residual-minimizing solutions, computation time as a function of matrix size, and the distribution of residuals across multiple trials.

	\section{Preliminaries and Notations}
	
	In this section, we briefly summarize the notations, basic definitions, and structural properties used throughout the paper.
We adopt the following notational conventions. The spaces $\mathbb{R}^{n}$ and $\mathbb{C}^{n}$ denote the real and complex $n$-dimensional vector spaces, respectively. For $x = [x_1, x_2, \dots, x_n]^T \in \mathbb{C}^n$, define
$
\|x\|_2 := \sqrt{\sum_{i=1}^{n} |x_i|^2},
$
where $|x_i|$ denotes the modulus of the complex number $x_i$.
For any matrix $A= [a_{ij}] \in \mathbb{C}^{n \times m}$, the Frobenius norm is defined by $\|A\|_{F} = \sqrt{\sum_{i,j} |a_{ij}|^{2}}$. Further, $A^T$ and $A^H$ denote the transpose and conjugate transpose of a matrix $A \in \mathbb{C}^{n \times m},$ respectively. 
The vectorization operator $\operatorname{vec}(\cdot)$ stacks the columns of a matrix $A = [a_{ij}] \in \mathbb{C}^{n \times m}$ into a single vector, i.e.,
\[
\operatorname{vec}(A) = [a_{11}, a_{21}, \dots, a_{n1}, a_{12}, \dots,a_{n2},\ldots,a_{1m},\ldots, a_{nm}]^{T}.
\]

The inverse, $\operatorname{vec}^{-1}(.)$, reshapes a vector $v \in \mathbb{R}^{mn}$ back into an $m \times n$ matrix. We also use the Kronecker product, denoted by $A \otimes B$, which constructs a block matrix formed by multiplying each entry of $A$ with the matrix $B$.

\subsection{Hankel Structure Matrix}
	A matrix $H \in \mathbb{C}^{n \times n}$ is called a \emph{Hankel matrix} if its entries are constant along every anti-diagonal, i.e.,
\[
h_{ij} = c_{i+j-1},
\]
for some vector $c = (c_{1}, c_{2}, \dots, c_{2n-1})^{T}$.  
Thus every Hankel matrix is determined by its generating vector $c$.
The collection of all $n\times n$ Hankel matrices forms a linear subspace of 
$\mathbb{F}^{\,n\times n}$ of dimension $2n-1$, where $\mathbb{F}\in\{\mathbb{R},\mathbb{C}\}$. 
This follows from the fact that a Hankel matrix is uniquely specified by its 
$2n-1$ anti-diagonal entries.
 
To encode this structure linearly, we introduce the matrix
$S \in \{0,1\}^{\,n^{2} \times (2n-1)},
$
which maps the vector of anti\text{-}diagonal values to the vectorized
Hankel matrix.  
For any generating vector $c \in \mathbb{C}^{2n-1}$, the associated Hankel
matrix $H$ satisfies
\[
\operatorname{vec}(H) = S\,c.
\]

Similarly, any Hankel perturbation $\Delta H$ can be written as
\[
\operatorname{vec}(\Delta H) = S z,
\qquad z \in \mathbb{C}^{2n-1},
\]

where the matrix $S$ is defined entrywise by
\[
S_{(j-1)n + i,\; k} =
\begin{cases}
	1, & \text{if } k = i + j - 1,\\[1mm]
	0, & \text{otherwise},
\end{cases}
\qquad 1 \le i,j \le n,\; 1 \le k \le 2n-1.
\]
Thus, the $k$th column of $S$ has ones exactly at the locations belonging to the
$k$th anti\text{-}diagonal of the matrix, and zeros elsewhere.  
Consequently, $Sz$ is precisely the vectorized Hankel matrix whose
anti\text{-}diagonals are given by~$z$.
Further, a Hankel matrix $H$ is said to be \emph{nonnegative} if
	$
	h_{ij} \ge 0 \quad \text{for all } i,j.
	$
	Since the entries on each anti-diagonal are identical, the nonnegativity condition becomes
	\[
	Sz \; \ge \; -\operatorname{vec}(H),
	\]
	which is a set of linear inequality constraints in the vector $z$.
	
	A complex scalar $\lambda$ and vector $x \ne 0$ form an eigenpair of a matrix $A$ if
	\[
	A x = \lambda x.
	\]
	In our setting, we seek a perturbation $\Delta H$ such that
	\[
	(H + \Delta H)x = \lambda x.
	\tag{$\star$}
	\]
	Using vectorization and Kronecker products, this condition becomes a system of linear equalities:
	\[
	(\operatorname{vec}(\Delta H)) = S z,\qquad 
	(x^{T} \otimes I_{n})S z = \lambda x - Hx,
	\]
	which must be satisfied exactly when an exact feasible solution exists.
	
	\subsection*{Nearest Nonnegative Hankel Matrix}
	
	\textbf{Definition}: Given a Hankel matrix $H \in \mathbb{C}^{n \times n}$ and an approximate eigenpair $(\lambda, x)$ with $x \neq 0$, the \textbf{nearest nonnegative Hankel matrix} $\widehat H$ is defined as follows:
	
	\begin{itemize}
		\item \textbf{Stage A (Exact Eigenpair Feasible):} If there exists a perturbation $\Delta H$ such that $(H + \Delta H)x = \lambda x$ and $H + \Delta H \ge 0$ while preserving the Hankel structure, the nearest matrix is obtained by solving
		\[
		\widehat H = H + \Delta H, \quad
		\Delta H = \arg \min_{\Delta H} \|\Delta H\|_F \quad 
		\text{s.t. } \widehat H x = \lambda x, \quad \widehat H \ge 0, \quad \widehat H \text{ Hankel}.
		\]
		This represents the \emph{structured backward error} for the eigenpair $(\lambda, x)$.
		
		\item \textbf{Stage B (Exact Eigenpair Infeasible):} If no such $\Delta H$ exists, the nearest admissible matrix is defined as the solution to a constrained least-squares problem
		\[
		\widehat H = H + \Delta H, \quad
		\Delta H = \arg \min_{\Delta H} \| (H + \Delta H)x - \lambda x \|_2 \quad 
		\text{s.t. } \widehat H \ge 0, \quad \widehat H \text{ Hankel}.
		\]
		Here, $\widehat H$ is the nonnegative Hankel matrix that minimizes the eigenpair residual in Frobenius or Euclidean norm.
	\end{itemize}
    
	\begin{remark}
In control theory, signal processing, and moment problems, a nonnegative Hankel 
matrix corresponds to a realizable moment sequence. In practice, however, complex 
eigenvalues may arise due to noise or modeling errors. Consequently, one may 
encounter a corrupted or physically invalid Hankel structure. In such cases, it is 
desirable to compute a valid Hankel matrix that is closest (in the least--squares 
sense) to the perturbed data. Stage~B therefore provides the best physically 
realizable approximation.
\end{remark}

	With the preliminaries and notations in place, we now address the problem of computing the nearest nonnegative Hankel matrix that realizes a given eigenpair.

\section{Main theorem \& its proof}

\begin{theorem}
	Let $H\in\mathbb{R}^{n\times n}$ be a Hankel matrix, let
	$\lambda\in\mathbb{C}$ and $x\in\mathbb{C}^n$ be a nonzero vector.
	Let $S\in\{0,1\}^{n^2\times(2n-1)}$ (defined previously) be the Hankel structure 
matrix, and let $z\in\mathbb{R}^{2n-1}$ denote the vector of Hankel parameters so that 
$\operatorname{vec}(\Delta H)=S z$ parameterizes all Hankel perturbations.

	Define the residual
	\[
	r := \lambda x - Hx,
	\]
	and
	\[
	C := (I_n\otimes x^T) S \in\mathbb{C}^{n\times(2n-1)}.
	\]
	Consider the two optimization problems used by the algorithm:
	
	\begin{enumerate}
		\item[{\bf (1)}] Stage A: minimum-norm exact correction
		\[
		\begin{aligned}
			&\min_{z\in\mathbb{R}^{2n-1}} \; \|S z\|_2^2\\
			&\text{subject to}\quad C z = r,\qquad S z \ge -\operatorname{vec}(H).
		\end{aligned}
		\tag{A}
		\]
		
		\item[{\bf (2)}] Stage B: constrained least-squares residual minimization
		\[
		\begin{aligned}
			&\min_{z\in\mathbb{R}^{2n-1}} \; \|C z - r\|_2^2\\
			&\text{subject to}\quad S z \ge -\operatorname{vec}(H).
		\end{aligned}
		\tag{B}
		\]
	\end{enumerate}
	
	Then the following statements hold.
	
	\begin{itemize}
		\item[(i)] \textbf{Feasibility equivalence.} There exists a nonnegative Hankel matrix
		$\widehat H = H+\Delta H$ with $\widehat H x = \lambda x$ if and only if
		the linear feasibility system
		\[
		C z = r,\qquad S z \ge -\operatorname{vec}(H)
		\]
		is solvable. In that case any solution $z$ produces
		$\Delta H = \operatorname{vec}^{-1}(S z)$ and $\widehat H = H+\Delta H$ satisfies the constraints.
		
		\item[(ii)] \textbf{Stage A optimality.} If (A) is feasible then any global minimizer
		$z^\star$ of (A) yields a perturbation $\Delta H^\star=\operatorname{vec}^{-1}(S z^\star)$
		with the following properties:
		\begin{enumerate}
			\item $\widehat H = H + \Delta H^\star$ is nonnegative and satisfies $\widehat H x = \lambda x$.
			\item $\Delta H^\star$ has minimum Frobenius norm among all Hankel perturbations that
			produce a nonnegative $\widehat H$ with the same exact eigenpair, i.e.
			\[
			\|\Delta H^\star\|_F = \min\{ \|\Delta H\|_F : (H+\Delta H)x=\lambda x,\; H+\Delta H\ge 0,\; \Delta H\ \text{Hankel}\}.
			\]
			\item The problem (A) is a convex quadratic program; a global minimizer exists. Uniqueness of $z^\star$ (hence of $\Delta H^\star$) holds when the quadratic form $z\mapsto z^T(S^T S)z$ is strictly convex on the affine feasible set (for example if $S^T S$ is positive definite on the nullspace of the active equalities).
		\end{enumerate}
		
		\item[(iii)] \textbf{Stage B correctness and interpretation.} If (A) is infeasible, then (B) is solved by the algorithm and any minimizer $\tilde z$ of (B) yields
		\[
		\widetilde{\Delta H} = \operatorname{vec}^{-1}(S\tilde z),\qquad \widehat H = H + \widetilde{\Delta H},
		\]
		which is a nonnegative Hankel matrix that \emph{minimizes the eigenpair residual}
		\[
		\|\widehat H x - \lambda x\|_2 = \|C \tilde z - r\|_2
		\]
		among all nonnegative Hankel matrices. The optimization (B) is a convex quadratic program and therefore admits a global minimizer; uniqueness holds under the standard strict convexity condition on $C^T C$ restricted to the feasible set.
		
		\item[(iv)] \textbf{Relationship of Stage A and Stage B.} 
        If (A) is feasible, then any minimizer $z^\star$ of (A) attains zero
residual and therefore yields the smallest possible perturbation $\Delta H^\star$.
If (A) is infeasible, then (B) returns the best least--squares approximation to the
eigenpair under the Hankel and nonnegativity constraints, but its solution does not,
in general, minimize $\|\Delta H\|_F$.

	\end{itemize}
\end{theorem}

\begin{proof}
	Most statements follow from straightforward algebra and convexity arguments; we give the concise proof.
	
	\medskip\noindent\textbf{(i) Feasibility equivalence.}
	Suppose first that there exists a nonnegative Hankel matrix $\widehat H=H+\Delta H$
	with $\widehat H x=\lambda x$. Since $\Delta H$ is Hankel there exists $z$ with
	$\operatorname{vec}(\Delta H)=S z$, and then
	\[
	(\Delta H)x = (I_n\otimes x^T)\operatorname{vec}(\Delta H) = (I_n\otimes x^T) S z = C z.
	\]
	From $\widehat H x=\lambda x$ we obtain $(\Delta H)x = \lambda x - Hx = r$, hence
	$C z = r$. Nonnegativity $\widehat H \ge 0$ is equivalent to $\operatorname{vec}(H)+S z\ge 0$,
	i.e. $S z \ge -\operatorname{vec}(H)$. Thus $z$ satisfies the feasibility system.
	
	Conversely, if a vector $z$ satisfies $C z=r$ and $S z\ge -\operatorname{vec}(H)$,
	then defining $\Delta H=\operatorname{vec}^{-1}(S z)$ yields a Hankel perturbation
	for which $\widehat H=H+\Delta H$ is nonnegative and satisfies $\widehat H x=\lambda x$.
	
	\medskip\noindent\textbf{(ii) Stage A optimality.}

Let $z^\star$ be any global minimizer of \textup{(A)}, and define
\[
\Delta H^\star := \operatorname{vec}^{-1}(S z^\star),
\qquad
\widehat H := H + \Delta H^\star.
\]

\emph{(1) Feasibility properties.}
Because $z^\star$ satisfies all constraints of (A), we first note that
\[
C z^\star = r = \lambda x - Hx.
\]
Using the identity
\[
C = (I_n \otimes x^T) S,
\]
we have
\[
C z^\star
=
(I_n \otimes x^T)\operatorname{vec}(\Delta H^\star)
=
\Delta H^\star x.
\]
Hence
\[
\Delta H^\star x = \lambda x - Hx,
\qquad\Longrightarrow\qquad
(H+\Delta H^\star)x = \lambda x.
\]
Thus the exact eigenpair condition is satisfied.

Next, the componentwise inequality constraint in (A),
\[
S z^\star \ge -\operatorname{vec}(H),
\]
implies
\[
\operatorname{vec}(\widehat H)
= 
\operatorname{vec}(H) + S z^\star
\ge 0,
\]
and therefore
\[
\widehat H = H + \Delta H^\star \ge 0.
\]
This proves that $\widehat H$ satisfies both the nonnegativity constraint
and the exact eigenpair constraint.

\medskip
\emph{(2) Minimal Frobenius-norm property.}
Since every Hankel perturbation is parameterized as
\(
\Delta H = \operatorname{vec}^{-1}(S z),
\)
the Frobenius norm satisfies
\[
\|\Delta H\|_F = \|\operatorname{vec}(\Delta H)\|_2 = \|S z\|_2.
\]
The objective function of (A) is precisely
\[
\min_{z}\;\|S z\|_2^2
=
\min_{\Delta H}\;\|\Delta H\|_F^2,
\]
where the minimization is restricted to Hankel perturbations $\Delta H$
that satisfy
\[
(H+\Delta H)x=\lambda x,
\qquad 
H+\Delta H \ge 0.
\]

Thus the feasible set of (A) is exactly the set of all Hankel perturbations
that produce a nonnegative matrix $\widehat H$ sharing the exact eigenpair
$(\lambda,x)$.  
Since $z^\star$ minimizes $\|S z\|_2$ over this set, the corresponding
$\Delta H^\star$ satisfies
\[
\|\Delta H^\star\|_F
=
\min\left\{
\|\Delta H\|_F \;:\;
(H+\Delta H)x = \lambda x,\;
H+\Delta H \ge 0,\;
\Delta H\ \text{Hankel}
\right\}.
\]

\medskip
\emph{(3) Existence and uniqueness.}
The objective function of (A) is
\[
f(z) = \tfrac{1}{2} z^T (S^T S) z,
\]
which is a convex quadratic form because $S^T S$ is positive semidefinite.
All constraints of (A) are affine.  
Hence (A) is a convex quadratic program, and therefore admits a global
minimizer.

For \emph{Uniqueness},
we assume that the quadratic form $z \mapsto z^\top(S^\top S)z$ is strictly
convex on the affine feasible set of (A), or equivalently, that
$S^\top S$ is positive definite on the nullspace of the active equality
constraints.

Let $z_1$ and $z_2$ be two minimizers of (A).  
Since the constraints are affine, the feasible region is convex.
Define $d := z_2 - z_1$.  
Because $Cz_1 = r = Cz_2$, we have
\[
C d = 0,
\]
so $d$ lies in the nullspace of the active equalities.

Consider the objective along the line segment 
$z(t) = z_1 + t d$ for $t\in[0,1]$.
A direct expansion gives
\[
f(z(t)) 
= \tfrac12 (z_1 + t d)^\top (S^\top S)(z_1 + t d)
= f(z_1) + t\, d^\top(S^\top S)z_1 
  + \tfrac12 t^2\, d^\top (S^\top S)d.
\]
Since both $z_1$ and $z_2$ are minimizers, the derivative at $t=0$
must vanish:
\[
\left.\tfrac{d}{dt} f(z(t))\right|_{t=0}
= d^\top (S^\top S) z_1 = 0.
\]
Hence
\[
f(z(t)) - f(z_1)
= \tfrac12 t^2\, d^\top (S^\top S)d.
\]

If $d\neq 0$, then $d^\top(S^\top S)d > 0$ by positive definiteness of
$S^\top S$ on $\ker(C)$, implying that $f(z(t)) > f(z_1)$ for all
$t\in(0,1]$, contradicting the assumption that $z_2$ is also a
minimizer.  
Therefore $d = 0$, and thus $z_1 = z_2$.

Consequently, the minimizer $z^\star$ of (A) is unique, and so is the
corresponding perturbation $\Delta H^\star = \operatorname{vec}^{-1}(S
z^\star)$.

\hfill$\square$
	
	\medskip\noindent\textbf{(iii) Stage B correctness.}
	If (A) is infeasible, no $z$ satisfies $Cz=r$ together with $Sz\ge -\operatorname{vec}(H)$.
	The algorithm then solves (B), whose objective is $\|C z - r\|_2^2 = \|(H+\operatorname{vec}^{-1}(Sz))x - \lambda x\|_2^2$,
	i.e. the squared eigenpair residual of the resulting matrix. Again (B) is a convex QP
	with closed feasible set, so a global minimizer $\tilde z$ exists. By construction
	$\tilde z$ yields a Hankel perturbation $S\tilde z$ satisfying the nonnegativity
	constraint and minimizing the eigenpair residual over the feasible set. Uniqueness
	follows under the same strict convexity conditions on $C^T C$ restricted to the feasible set.
	
	\medskip\noindent\textbf{(iv) Relationship.}
	If (A) is feasible, then there exists $z$ such that
\[
Cz = r.
\]
For such $z$, the residual in (A) satisfies
\[
\|Cz - r\|_2 = 0,
\]
which is the minimum possible value of the objective function in (A).
Thus any minimizer $z^\star$ of (A) satisfies
\[
Cz^\star = r,
\]
so Stage~A returns a perturbation $\Delta H^\star$ that enforces the
eigenpair constraints exactly.  
Because (A) also minimizes the quadratic form
\[
\|\Delta H\|_F^2 = \|Sz\|_2^2,
\]
over all feasible $z$, the perturbation $\Delta H^\star$ has minimum Frobenius norm
among all perturbations that produce an exact eigenpair. Hence Stage~A is both
\emph{exact} and \emph{norm‐minimal}.

If (A) is infeasible, then no $z$ satisfies $Cz=r$ under the Hankel and
nonnegativity constraints. In this case, Stage~B solves the least-squares problem
\[
\min_{z \in \mathcal{F}} \|Cz - r\|_2^2,
\]
where $\mathcal{F}$ denotes the feasible set defined by the constraints of (B).
Stage~B therefore returns the perturbation that gives the smallest possible
residual under these constraints. However, because (B) does not minimize
$\|Sz\|_2$ (the Frobenius norm of $\Delta H$), its solution does not necessarily
minimize $\|\Delta H\|_F$.

	This completes the proof.
\end{proof}

	Building on the theoretical results established in the main theorem, we now present a practical algorithm to compute the nearest nonnegative Hankel matrix for a given eigenpair.
	\newpage 
\begin{algorithm}
	\caption{Two--Stage Algorithm for Nearest Nonnegative Hankel Matrix with Prescribed Eigenpair}
	\label{Algo:TwoStage}
	\begin{algorithmic}[1]
		
		\Require Hankel matrix $H\in\mathbb{R}^{n\times n}$, target eigenpair $(\lambda, x)$
		\Ensure Corrected Hankel matrix $\widehat H$ and perturbation $\Delta H$
		
		\State Construct structure matrix $S$ such that $\operatorname{vec}(\Delta H)=Sz$
		\State Define nonnegativity constraint: $Sz \ge -\operatorname{vec}(H)$
		
		\State Compute eigenpair residual $r = \lambda x - Hx$
		\State Form $C = (I_n \otimes x^T) S$
		
		\Statex \vspace{0.15cm}
		\Statex {\bf Stage A (Exact Feasible Minimum--Norm Correction)}
		
		\State Attempt to solve the feasibility system
		\[
		C z = r,\qquad Sz \ge -\operatorname{vec}(H)
		\]
		(using linear/least--norm programming)
		
		\If{a feasible $z_A$ exists}
		\State $\Delta H = \operatorname{vec}^{-1}(S z_A)$
		\State $\widehat H = H + \Delta H$
		\State \Return $\widehat H,\, \Delta H$
		\EndIf
		
		\Statex \vspace{0.15cm}
		\Statex {\bf Stage B (Constrained Least--Squares Residual Minimization)}
		
		\State Solve the quadratic program
		\[
		z_B = \arg\min_z \| C z - r \|_2 
		\quad\text{s.t.}\quad Sz \ge -\operatorname{vec}(H)
		\]
		\State $\Delta H = \operatorname{vec}^{-1}(S z_B)$
		\State $\widehat H = H + \Delta H$
		
		\State \Return $\widehat H,\, \Delta H$
		
	\end{algorithmic}
\end{algorithm}

Next, we present numerical experiments that demonstrate the performance of the proposed algorithm, covering both exact and inexact eigenpair cases. We also include several graphical results that clearly illustrate the numerical behaviour of our algorithm, demonstrating its ability to distinguish between feasible and infeasible eigenpairs, its stability through residual clustering, and its computational performance across different matrix sizes.

	
%
%
%
%
%
%
%
%
%

\section{Numerical Experiments and Performance Analysis}

\subsection*{Example-1: Exact Nonnegative Hankel Matrix Exists}

Let the random Hankel matrix $H$ and its approximate eigenpair $(\lambda,x)$ be given as follows: 
\[
H =
\begin{bmatrix}
	0.0903 & -1.4229 &  0.0694 & -0.0928 &  1.4677 \\
	-1.4229 &  0.0694 & -0.0928 &  1.4677 &  1.3898 \\
	0.0694 & -0.0928 &  1.4677 &  1.3898 & -0.3613 \\
	-0.0928 &  1.4677 &  1.3898 & -0.3613 & -0.5182 \\
	1.4677 &  1.3898 & -0.3613 & -0.5182 &  0.1594
\end{bmatrix}
\]

\[
\lambda = 1.301415,\quad
x = \begin{bmatrix}
	-0.2392 &
	-0.5470 &
	-0.7013 &
	-0.3812 &
	\;\;0.0806
\end{bmatrix}^T.
\]

By using Algorithm~\ref{Algo:TwoStage},we get

	\[
	\widehat{H} =
	\begin{bmatrix}
		0.7482 & 0      & 0      & 0.5342 & 0.8847 \\
		0      & 0      & 0.5342 & 0.8847 & 0      \\
		0      & 0.5342 & 0.8847 & 0      & 0      \\
		0.5342 & 0.8847 & 0      & 0      & 1.4354 \\
		0.8847 & 0      & 0      & 1.4354 & 10.7154
	\end{bmatrix}.
	\]
	
 Further, the minimum Frobenius norm of the perturbation is
\[
\|\Delta H\|_F = 1.161675 \times 10^{1}.
\]
Clearly, 
\[
\| \widehat{H}x - \lambda x \|_2 = 6.976153 \times 10^{-15}.
\]

This example illustrates that the algorithm can successfully produce a matrix nearest to the given Hankel matrix with prescribed eigenpair while preserving Hankel structure and nonnegativity.

\subsection*{Example-2: No Exact Nonnegative Hankel Matrix}

Let the random Hankel matrix $H$ and its approximate eigenpair $(\lambda,x)$ be given as follows: 
\[
H =
\begin{bmatrix}
	-1.1923 &  1.2917 &  0.3320 & -0.6178 &  0.4433 \\
	1.2917 &  0.3320 & -0.6178 &  0.4433 & -0.8108 \\
	0.3320 & -0.6178 &  0.4433 & -0.8108 & -0.0862 \\
	-0.6178 &  0.4433 & -0.8108 & -0.0862 &  1.7344 \\
	0.4433 & -0.8108 & -0.0862 &  1.7344 &  0.6217
\end{bmatrix}.
\]

\[
\lambda = 0.152982, \quad
x =
\begin{bmatrix}
	0.7298 &
	0.2273 &
	0.6387 &
	-0.0018 &
	0.0881
\end{bmatrix}^T.
\]

Algorithm~\ref{Algo:TwoStage} tells that exact nonnegative Hankel matrix does not exist. In this case the algorithm computes $\widehat H$ such that $\| \widehat H x - \lambda x \|_2$ is minimum:
\[
\widehat H =
\begin{bmatrix}
	0       & 0.2849 & 0.0746 & 0      & 0      \\
	0.2849  & 0.0746 & 0      & 0      & 0      \\
	0.0746  & 0      & 0      & 0      & 0.1552 \\
	0       & 0      & 0      & 0.1552 & 0      \\
	0       & 0      & 0.1552 & 0      & 0
\end{bmatrix}.
\]
The Frobenius norm of the perturbation is
	\[
	\|\Delta H\|_{F} = 3.919675.
	\]
	Clearly, it can be seen that 
	\[
	\| \widehat H x - \lambda x \|_2 = 2.106211 \times 10^{-1}.
	\]

This example illustrates that while an exact nonnegative Hankel matrix with the prescribed eigenpair does not exist, the algorithm successfully computes $\widehat H$ such that $\| \widehat H x - \lambda x \|_2$ is minimum and  and $\widehat H$ preserving Hankel structure and nonnegativity.

\begin{remark}
The Frobenius norm $\|\Delta H\|_F$ measures the distance from the given matrix $H$
to the nearest nonnegative Hankel matrix that best satisfies the prescribed eigenpair.
In Example~1, an exact feasible solution exists, so $\widehat H = H + \Delta H$ satisfies
$(H+\Delta H)x = \lambda x$ and $\|\Delta H\|_F$ is the minimum possible among all such
nonnegative Hankel matrices.
In Example~2, no exact solution exists; hence Stage~B of Algorithm~\ref{Algo:TwoStage} computes a matrix $\widehat H$
that minimizes the residual $\|\widehat H x - \lambda x\|_2$ while keeping $\Delta H$
as small as allowed under the Hankel and nonnegativity constraints.
\end{remark}
 \subsection{Performance Analysis of Algorithm}

To further investigate the computational behavior of our algorithm, we conducted experiments on matrices of varying sizes $n = 10, 20, 30, \dots, 300 $. For each matrix size, we performed 10 trials by selecting arbitrary candidate eigenpairs and evaluating their feasibility. We recorded the residual norms and CPU times for each trial to analyze both convergence behavior and computational cost. 
Figure~\ref{fig:stagecompare} shows that infeasibility appears across all tested matrix sizes. 
For each dimension $n$, among the randomly generated eigenpairs, there always exists at least one eigenpair for which no nonnegative Hankel matrix realizes it as an exact eigenpair. 
Therefore, infeasibility is not dependent on the matrix size; rather, for almost every size $n$, one can find eigenpairs that are numerically incompatible with the nonnegative Hankel structure.

\begin{figure}[htbp]
    \centering
    \includegraphics[width=.85\textwidth]{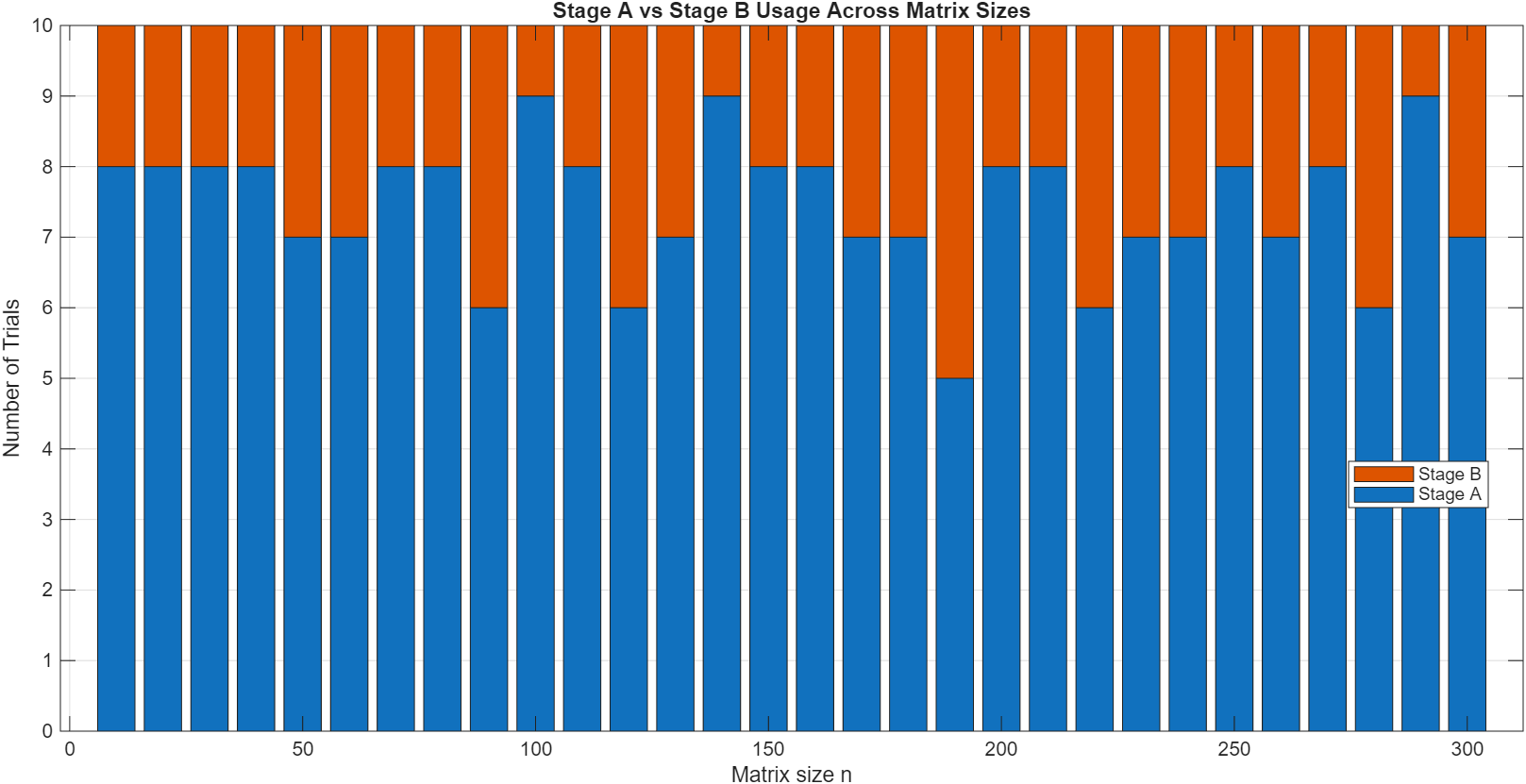}
    \caption{Stage A vs Stage B usage across matrix sizes. Blue denotes exact feasible cases; orange denotes infeasible cases requiring Stage B optimization.}
    \label{fig:stagecompare}
\end{figure}


Next, figure~\ref{fig:residual_distribution} displays the scatter plot of the eigenpair
residual versus the CPU time over all matrix sizes and all random trials.
Two clearly separated clusters are visible. The first cluster, located in the
lower-left corner, contains residuals between $10^{-16}$ and $10^{-8}$ with CPU
times below $2$ seconds. These points correspond to cases where the prescribed
eigenpair is feasible, i.e., Stage~A succeeds in finding a nonnegative Hankel
matrix satisfying the eigenpair relation exactly. The second cluster, appearing
in the upper-right region, contains residuals between $10^{0}$ and $10^{1}$ and
CPU times ranging from $3$ to $25$ seconds. These trials represent infeasible
eigenpairs for which Stage~A fails and the algorithm switches to Stage~B,
producing only an approximate solution with significantly larger residual and
computational cost. A small transition region lies between the two clusters.
Overall, the plot shows that residual magnitude and CPU time jointly diagnose
feasibility: feasible cases are solved quickly with machine-precision
residuals, whereas infeasible cases are consistently slower and yield large
residuals.


\begin{figure}[!ht]
    \centering
    \includegraphics[width=0.85\textwidth]
    {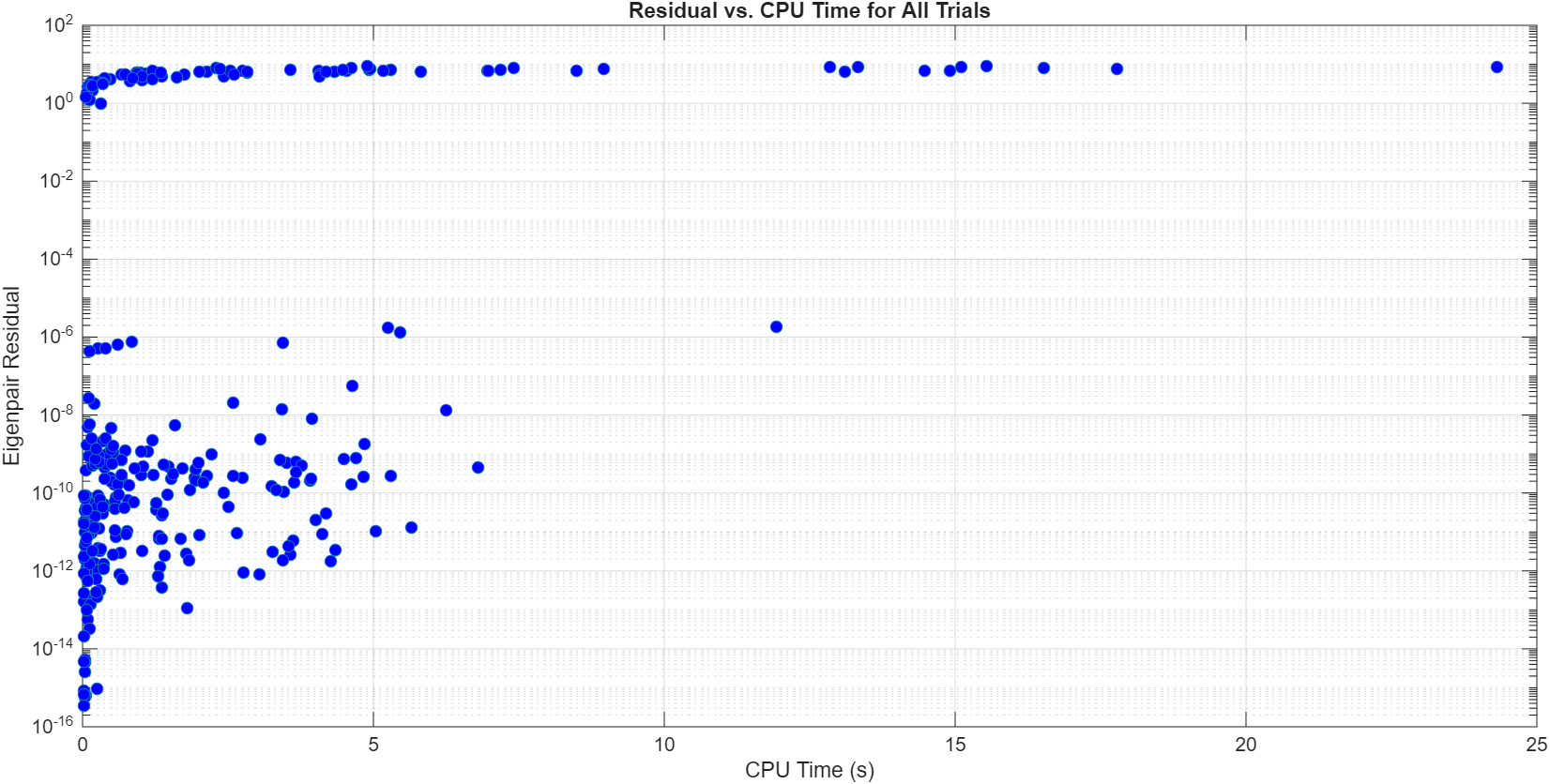}
    \caption{Residual distribution across all trials. Two distinct clusters appear:  
    (i) very small residuals ($10^{-16}$–$10^{-8}$) corresponding to feasible eigenpairs (Stage A),  
    (ii) large residuals ($\approx 1$–$10$) corresponding to infeasible eigenpairs requiring approximation (Stage B).}
    \label{fig:residual_distribution}
\end{figure}


Finally, Figure~\ref{fig:cputime_vs_size} illustrates how the average CPU time of the proposed 
Hankel eigenpair algorithm varies with the matrix size~$n$. As expected, the overall trend shows 
an increase in computational effort for larger matrices, since higher dimensions lead to more 
constraints and a more expensive quadratic program for \texttt{quadprog}.
However, the curve is not monotonic. The fluctuations arise because some trials terminate in 
\emph{Stage~A}, which is very fast, while others proceed to the more computationally demanding 
\emph{Stage~B}. 
A particularly important observation occurs at \textbf{$n = 250$}; despite the general upward 
trend in CPU time, the \emph{average time decreases sharply} at this point. This behaviour directly 
supports our theoretical understanding that is for certain matrix sizes the feasibility 
conditions are more frequently satisfied in \emph{Stage~A}, causing many trials to terminate early. 
As a result, the average time does \emph{not} increase with~$n$ at $n=250$ (and for many other different matrix sizes); instead, it decreases, 
producing a noticeable dip in the curve.
Thus, this irregularity is not an anomaly but a consequence of the algorithmic structure: the 
relative frequency of Stage~A versus Stage~B completions strongly influences the observed CPU times.

\begin{figure}[!ht]
    \centering
    \includegraphics[width=0.85\textwidth]{
    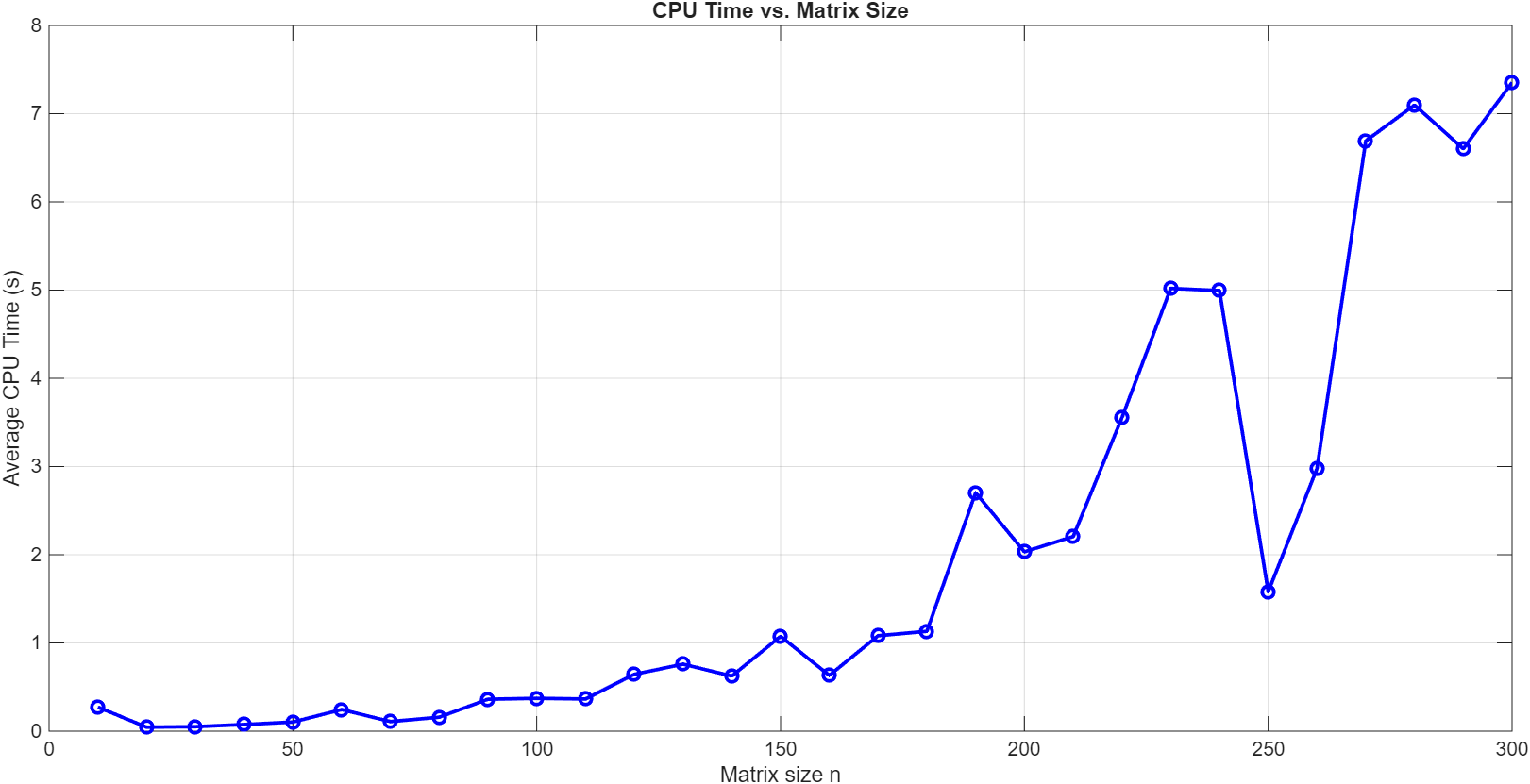}
    \caption{CPU time as a function of matrix size $n$. The overall growth reflects increased SDP complexity. Occasional spikes correspond to infeasible eigenpairs where Stage B dominates, confirming that infeasibility significantly increases the computational burden.}
    \label{fig:cputime_vs_size}
\end{figure}

\section*{Conclusion}

In this work, we studied the problem of computing minimal Frobenius-norm perturbations to a Hankel matrix so that a prescribed eigenpair is realized while preserving entrywise nonnegativity. We formalized the notion of a \emph{nearest nonnegative Hankel matrix} and highlighted the interplay between the linear Hankel structure and the convex nonnegativity constraints, which makes the problem nontrivial.

We developed a fully numerical framework that handles both feasible and infeasible cases. If the feasibility region is nonempty, the method computes the exact structured backward error via the minimum-norm perturbation. However, when no admissible perturbation satisfies the imposed structure, the framework switches to a constrained least-squares formulation to identify the closest matrix that minimizes the corresponding eigenpair residual. The approach leverages modern convex optimization techniques, including linear and quadratic programming, and uses the Hankel structure matrix to efficiently parameterize admissible perturbations.

The main contribution of this article is to provide a unified theoretical and computational framework for assessing eigenpair sensitivity under nonnegativity-preserving Hankel perturbations. Moreover, it also gives the explicit characterization of feasible perturbations and practical algorithms for both exact and approximate settings. Numerical experiments demonstrate the effectiveness of the proposed approach, illustrating both cases where the prescribed eigenpair is exactly realizable and cases where it is not, along with graphical analyses of perturbation norms, eigenpair residuals, and computational performance across different matrix sizes.

The work in this paper provides a concrete, implementable method for structured eigenpair analysis in the case of Hankel matrices with positivity constraints and hence bridges a gap between structured inverse eigenvalue problems, nonnegative matrix analysis, and practical computational methods.

\section*{Authorship contribution statement}

{\bf Prince Kanhya:} Conceptualization, Methodology, Formal analysis, Investigation, Validation, Writing-original draft, Writing-review \&\ editing, Supervision.\\
{\bf Udit Raj:} Methodology, Formal analysis, Validation, Writing-original draft, Writing-review \&\ editing.

\section*{Funding}
{\bf Prince Kanhya} acknowledges financial support from the Research and Development (R\&D) Department, Indian Institute of Technology (IIT) Guwahati, Assam, India.

\section*{AI-Assistance Disclosure}
The authors used an artificial intelligence language model (ChatGPT, developed by OpenAI) for rephrasing, language polishing, and basic grammatical correction of certain sentences. All mathematical results, algorithms, proofs, and scientific conclusions presented in this manuscript are entirely authored, verified, and approved by the authors. 

\section*{Declaration of Competing Interest}
No potential conflict of interest was reported by the authors.

\end{document}